\documentclass[10pt]{amsart}
\usepackage{amsmath}
\usepackage{amssymb,amscd}
\usepackage{graphicx}
\usepackage{mathdots}
\usepackage{mathrsfs}       
\usepackage{graphicx}
\usepackage{bbm} 
\usepackage{microtype} 

\usepackage{colonequals} 
\usepackage{mathtools}

\usepackage{tikz,tkz-euclide}
\usepackage{tikz-cd,calc}
\usetikzlibrary{arrows}
\usepackage{tikz-3dplot}

\usepackage{color}\definecolor{darkblue}{rgb}{0,0.1,.5}
\usepackage[colorlinks=true,linkcolor=darkblue, urlcolor=darkblue, citecolor=darkblue]{hyperref}
\usepackage{cleveref}

\theoremstyle{plain}
\newtheorem{theorem}{Theorem}[section]
\newtheorem{lemma}[theorem]{Lemma}
\newtheorem{proposition}[theorem]{Proposition}
\newtheorem{corollary}[theorem]{Corollary}

\theoremstyle{definition}
\newtheorem{definition}[theorem]{Definition}
\newtheorem{example}[theorem]{Example}

\newtheorem{construction}[theorem]{Construction}

\theoremstyle{remark}
\newtheorem*{remark}{Remark}

\numberwithin{equation}{section}

\def \begineq{\begin{equation}}
\def \endeq{\end{equation}}

\def\Bier{\mathrm{Bier}}

\def\Td{\mathrm{Td}}
\def\sK{\mathcal K}

\def\N{\mathbb N}
\def\Z{\mathbb Z}
\def\Q{\mathbb Q}
\def\R{\mathbb R}
\def\C{\mathbb C}

\DeclareMathAlphabet{\mathbbmsl}{U}{bbm}{m}{sl}

\DeclareMathOperator{\sign}{sign}

\makeatletter
\@namedef{subjclassname@2020}{%
  \textup{2020} Mathematics Subject Classification}
\makeatother

\title[Characteristic numbers of canonical toric manifolds and their applications]{Characteristic numbers of canonical toric manifolds\\ and their applications}

\author[Gruji\'c]{Vladimir Gruji\'c}
\address{Faculty of Mathematics, University of Belgrade, Serbia}
\email{vladimir.grujic@matf.bg.ac.rs}

\author[Limonchenko]{Ivan Limonchenko}
\address{Mathematical Institute of the Serbian Academy of Sciences
and Arts (SASA), Belgrade, Serbia}
\email{ivan.limoncenko@turing.mi.sanu.ac.rs}



\begin{document}

\begin{abstract}
We compute all the Chern, Milnor and Pontryagin numbers for canonical toric manifolds associated with abstract simplicial complexes and the Stiefel--Whitney numbers for their real counterparts. Applications include combinatorial characterizations of the unitary, oriented and unoriented bordism classes, new geometrical representatives of the unitary bordism ring generators, a combinatorial criterion for a canonical toric manifold to bound, as well as the dimension estimates for their immersions into euclidean spaces.
\end{abstract}

\maketitle

\section{Introduction}

In 1992, a remarkable class of simplicial spheres $\mathrm{Bier}(\sK)$ was introduced~\cite{Bier}: each of them is constructed as the deleted join of an arbitrary simplicial complex $\sK$ on $[m]:=\{1,2,\ldots,m\}$, different from the entire simplex $\Delta_{[m]}$ with $m$ vertices, and its combinatorial Alexander dual complex $\sK^\vee$. Since then, Bier spheres have found significant applications in topological and geometrical combinatorics, in particular, those related to the proofs of the classical Van Kampen - Flores theorem, see~\cite{Matousek}, and $g$-theorem, see~\cite{BPSZ05}. Bier spheres also serve as a source of infinitely many non-polytopal spheres, however, no explicit example has been constructed so far. 

A recent construction of the canonical starshaped realization for any Bier sphere obtained in~\cite{JTZ,JZ} justified the importance of Bier spheres in toric geometry~\cite{CLS} and toric topology~\cite{TT}. A canonical characteristic map $\Lambda_\sK$ for each Bier sphere $\Bier(\sK)$ was introduced in~\cite{LS} and then the canonical toric manifolds and the canonical polyhedral products associated with abstract simplicial complexes were constructed and studied in~\cite{LTZ1, LTZ2}. 

In particular, the construction of this new wide class of nonsingular complete toric varieties, one canonical toric manifold $X_\sK$ for each simplicial complex $\sK$ different from the entire simplex, allowed for a short topological proof of the Dehn-Sommerville relations for Bier spheres, see also~\cite{BPSZ05}, and provided immediately the existence of an infinite family of toric manifolds that are not quasitoric, see also~\cite{Suyama}. 

In this paper, we continue the study of the topology of canonical toric manifolds $X_{\sK}$ and $X_{\sK}^\mathbb{R}$, determined by the canonical regular realizations of Bier spheres $\Bier(\sK)$, both in the complex and real cases, respectively. Most of the canonical toric manifolds $X_{\sK}$ are non-projective complete algebraic varieties, since the vast majority of them arise from the canonical fans $\Sigma_{\sK}$ whose underlying simplicial complexes are non-polytopal spheres. It follows from~\cite[Theorem 19, Corollary 20]{JZ} that $X_{\sK}$ is projective if and only if $\sK$ is a threshold complex. 

We calculate characteristic numbers of canonical toric manifolds. It turns out that the Chern numbers $c_I[X_{\sK}]$ depend only on the local combinatorics of the underlying simplicial complex $\sK$, which is encoded in the $\alpha$-vector of $\sK$. The latter combinatorial invariant of a simplicial complex is also interesting from the viewpoint of combinatorial commutative algebra: it was introduced and studied in~\cite{JM}, where it was shown that its components appear as coefficients of the exponential Hilbert polynomial of the Stanley--Reisner ring of $\sK$. 

Once we know the characteristic numbers, the bordism classes of the canonical toric manifolds are determined. One of our main results is Theorem \ref{eqcomplexcobordismh} in which we show that the complex bordism class of the canonical toric manifold $X_{\sK}$ is completely determined by the $h$-vector of the Bier sphere $\mathrm{Bier}(\sK)$. We provide a topological proof for the known formula, see~\cite{BPSZ05}, for the $h$-vector of $\mathrm{Bier}(\sK)$ in terms of the $f$-vector of $\sK$ using the Hirzebruch $\chi_y$-genus, which also leads us to the signature formula for $X_\sK$.

The strong structure property of the topology of canonical toric manifolds comes from the natural involution induced by the duality between a complex $\sK$ and its Alexander dual $\sK^\vee$. We prove that the $\alpha$-vectors of $\sK$ and $\sK^\vee$ are related by an involutive matrix $A$. While the rank of a complex bordism group equals the partition number $\dim\Omega^U_{2(m-1)}=\pi(m-1)$ and increases rapidly with $m$, the rank of its subgroup spanned by the $2(m-1)$-dimensional canonical toric manifold classes is equal to $\lfloor\frac{m+1}{2}\rfloor$. In Theorem~\ref{additive}, we show that the canonical toric manifolds $X_j:=X_{\sK_j}, \ 0\leq j\leq m-1$, where $\sK_{j}:=\Delta_{[j]}\ast\partial\Delta_{[m-j]}$ form a spanning set of this free abelian subgroup in the real dimension $2m-2$ and find its basis. Each complex manifold $X_j$ is diffeomorphic to the product of complex projective spaces, but with a nonstandard complex structure. By calculating the Milnor numbers, we determine which canonical toric manifolds may serve as representatives of the polynomial generators of the complex bordism ring $\Omega^U_\ast$. 

In the real case, the dependence of the Stiefel--Whitney and Pontryagin classes is more simple and involves only the Euler characteristic $\chi(\sK)$. We determine the unoriented and oriented bordism classes of the canonical toric manifolds. As a consequence, we obtain a large source of examples of manifolds which belong to the kernel of the forgetting morphisms $\Omega_\ast^U\rightarrow\Omega_\ast^{SO}$ and $\Omega_\ast^{SO}\rightarrow\Omega_\ast^O$. 

We also study the problem of immersions of the canonical real toric manifolds into euclidean spaces and provide the estimate of dimension of the euclidean space that allows immersions. In each dimension, we construct a family of canonical real toric manifolds on which the upper bound of the minimal dimension of the Euclidean space in which they can be immersed is achieved. 

\subsection*{Acknowledgements}
We are grateful to  Taras Panov, Ale\v{s} Vavpeti\v{c}, and Rade \v{Z}ivaljevi\'c for various fruitful discussions, valuable comments and suggestions. The first author is partially supported by the Ministry of  Science, Technological Development and Innovation, Republic of Serbia, through the project 451-03-33/2026-03/200104. The second author was supported by the Serbian Ministry of Science, Technological Development and Innovation through the Mathematical Institute of the Serbian Academy of Sciences and Arts.

\section{Basic definitions and constructions}

In this section, we recall some key definitions and results concerning simplicial complexes, Bier spheres, and toric manifolds.
We start with abstract simplicial complexes and rational polyhedral fans. In what follows, for any positive integer $m\in\N$, we use the notation $[m]$ for the set $\{1,2,\ldots,m\}$.

\begin{definition}
Let $\sK$ be a subset of the power set $2^{[m]}$. We call $\sK$ an (abstract) \emph{simplicial complex} on $[m]$ if the following condition holds:
$$
\sigma\in \sK\text{ and }\tau\subset\sigma \Longrightarrow \tau \in \sK.
$$
\end{definition}

Elements of $\sK$ are its simplices (or, faces). Maximal, with respect to inclusion, faces of $\sK$ are called its \emph{facets} and their set is denoted by $\max(\sK)$. Similarly, the set of \emph{minimal non-faces} of $\sK$ with respect to inclusion is denoted by $\min(\sK)$. Note that each of the sets $\max(\sK)$ and $\min(\sK)$ determines the simplicial complex $\sK$ on $[m]$. Recall that the \emph{dimension} of $\sK$ is one less than the maximal cardinality of a face of $\sK$; this number is denoted by $\dim \sK$.

We say that a singleton $\{i\}$ (or, simply $i$ itself) is a real (or, geometrical) \emph{vertex} of $\sK$ if $\{i\}\in \sK$; we call it a \emph{ghost vertex} otherwise. The set of real vertices of $\sK$ will be denoted by $V(\sK)$. 

\begin{example}
Given any $m\in\N$ we consider the following simplicial complexes $\sK$ on $[m]$:
\begin{itemize}
    \item $\Delta_{[m]}:=2^{[m]}$, the simplex. It has $\dim(\sK)=m-1$ and $V(\sK)=[m]$;
    \item $\partial\Delta_{[m]}:=2^{[m]}\setminus \{[m]\}$, the boundary of a simplex. It has $\dim \sK=m-2$ and $V(\sK)=[m]$;
    \item 
    $\varnothing_{[m]}:=\{\varnothing\}$, the void complex. It has $\dim \sK=-1$ and $V(\sK)=\varnothing$.
\end{itemize}
\end{example}

In convex geometry, the analogue of a simplicial complex is called a fan. Geometrically, a simplicial complex is a collection of simplices and, similarly, a fan is a collection of cones. 

\begin{definition}
Consider the standard integer lattice $\Z^n\subset\R^n$. Each collection of vectors $S=\{v_1,\ldots,v_s\}\subset\Z^n$ generates a (rational polyhedral) \emph{cone} in $\R^n$ as follows:
$$
\sigma=\R_{\ge}\langle S\rangle:=\{r_1v_1+\cdots+r_sv_s\,|\,r_i\geq 0\text{ for each }1\leq i\leq s\}\subseteq\R^n.
$$
In this paper, we consider only \emph{strongly convex} cones $\sigma$; that is, $\sigma$ does not contain any line in $\R^n$. A cone is called \emph{simplicial} (respectively, \emph{nonsingular}) if it is generated by a part of a basis of $\R^n$ (respectively, $\Z^n$).
\end{definition}

\begin{definition}\label{FanDef}
A (rational polyhedral) \emph{fan} in $\R^n$ is a set $\Sigma$ of cones in $\R^n$ such that each face of a cone in $\Sigma$ is again a cone of $\Sigma$ and the intersection of any two cones in $\Sigma$ is a face of each of them. A fan $\Sigma$ in $\R^n$ is called
\begin{itemize}
\item \emph{complete}, if the union of all its cones equals $\R^n$;
\item \emph{simplicial}, if all its cones are simplicial;
\item \emph{regular}, if all its cones are nonsingular.
\end{itemize}
\end{definition}

\begin{example}\label{NormalFanExample}
Suppose $P$ is a lattice simple polytope in $\R^n$.  Then the fan $\Sigma_P$ whose maximal (by inclusion) cones are generated by the primitive integer normal vectors to facets intersecting at a vertex of $P$ is called a \emph{normal fan} of $P$. Obviously, it is a complete simplicial fan. We say that $P$ is a \emph{Delzant polytope} if its normal fan is regular.
\end{example}

Finally, we introduce a construction that relates fans to simplicial complexes. Suppose $\Sigma$ is a simplicial fan and consider the $m$-tuple $(v_1,\ldots,v_m)$ of primitive integer generators of the 1-dimensional cones of $\Sigma$. This determines the \emph{underlying simplicial complex} $\mathcal K_\Sigma$ on $[m]$ as follows: by definition, $\{i_i,\ldots,i_p\}\in \mathcal K_\Sigma$ if and only if the vectors $v_{i_1},\ldots,v_{i_p}$ generate a cone in $\Sigma$. In particular, $\Sigma$ is a complete simplicial fan if and only if $\mathcal K_\Sigma$ is a starshaped sphere, see~\cite{TT}. 

We say that a simple polytope $P$ has a \emph{Delzant realization} if and only if its \emph{nerve complex} $K_P=\partial P^*$ is combinatorially equivalent to $\mathcal K_\Sigma$, where $\Sigma$ is a complete regular fan. On the other hand, not every simple polytope has a Delzant realization even in dimension $3$: the dodecahedron can serve as a counterexample, since it has neither triangular nor quadrangular faces, see~\cite{Delaunay}. Hence, more generally, we say that a simplicial sphere $\sK$ has a \emph{regular realization} if it is isomorphic to the underlying complex of a complete regular fan.

Now we are going to introduce the main combinatorial object of interest in this paper, the Bier sphere. Let $
\sK\neq\Delta_{[m]}$ be a simplicial complex on $[m]$ with $m\geq 2$. In what follows, we consider the copy $[m']:=\{1',2',\ldots,m'\}$ of the set $[m]$ such that $[m]\cap [m']=\varnothing$. 

\begin{definition}
The (combinatorial) \emph{Alexander dual} of $\sK$ is a simplicial complex $\sK^\vee$ on $[m']$ such that 
$$
I'\in \max(\sK^\vee)\Longleftrightarrow I^c\in\min(\sK).
$$
That is, (maximal) faces of $\sK^\vee$ correspond to complements of (minimal) non-faces of $\sK$.

The Bier sphere $\Bier(\sK)$ of $\sK$ is a simplicial complex on $[m]\sqcup [m']$ such that
$$
I\sqcup J'\in\Bier(\sK)\Longleftrightarrow I\in \sK, J'\in \sK^\vee\text{ and }I\cap J=\varnothing.
$$
That is, $\Bier(\sK)$ is a \emph{deleted join} $\Bier(\sK):=\sK\ast_\delta \sK^\vee$ of $\sK$ and its Alexander dual complex $\sK^\vee$. 
\end{definition}

The $(m-2)$-dimensional simplicial sphere $\mathrm{Bier}(\sK)$ has a regular realization in $\mathbb{R}^{m-1}$ as follows.
Denote by $(e_1,\ldots,e_{m-1})$ the standard basis in $\R^{m-1}$ and set $e_{m}:=-(e_1+\cdots+e_{m-1})$.

\begin{construction}[\cite{LTZ1}]
Let $I, J\subset [m]$ and $I\cap J=\varnothing$. We consider the following sets:
$$
G(I,J):=\{-e_{i}, e_{j}\,|\,i\in I, j\in J\},
C(I,J):=\R_{\ge}\langle G(I,J)\rangle,\text{ and }\Sigma_\sK:=\{C(I,J)\,|\,I\in \sK, J'\in \sK^\vee\}.
$$
\end{construction}

\begin{example}
For each $m\geq 2$,  $\Bier(\varnothing_{[m]})\cong\Bier(\partial\Delta_{[m]})$ is the boundary of the simplex $\Delta^{m-1}$. Furthermore, $\sK=\varnothing_{[m]}$ yields a complete fan $\Sigma_\sK$ whose rays are generated by the vectors from the set $\{e_1,\ldots,e_{m-1},e_m\}$. The toric variety corresponding to this fan is the complex projective space and we obtain $X_{\sK}=\C P^{m-1}$.
\end{example}

If $\sK\neq\Delta_{[m]}$ is a simplicial complex on $[m]$ with $m\geq 2$, it was shown in~\cite{LTZ2} that $\Sigma_\sK$ is a complete regular fan whose underlying simplicial complex is isomorphic to the Bier sphere $\Bier(\sK)$. By definition, the rays of the \emph{canonical fan} $\Sigma_\sK$ described above are the columns of the \emph{canonical characteristic matrix} $\Lambda_\sK$. In the framework of toric geometry~\cite{CLS}, these data give rise to the \emph{canonical toric manifold} $X_\sK$ and its real counterpart $X^{\R}_\sK$. 

Recall that the \emph{Stanley--Reisner ring} (or, a \emph{face ring}) of a simplicial complex $\sK$ with $m$ vertices over a ring with unit $\Bbbk$ is defined as a quotient ring
$$
\Bbbk[\sK]:=\Bbbk[x_1,\ldots,x_m]/\mathcal I,
$$
where the \emph{Stanley--Reisner ideal} (or, \emph{face ideal}) of $\sK$ is generated by all square-free monomials of the form $x_{i_1}\cdots x_{i_r}$ with $\{i_1,\ldots,i_r\}\notin \sK$.

The integer cohomology ring of the canonical toric manifold $X_\sK$ is determined by the Danilov-Jurkiewicz theorem, see~\cite{Dan}, as the quotient ring
\begin{center} 
$H^*(X_\sK;\mathbb{Z})\cong\Z[\Bier(\sK)]/\mathcal J=\mathbb{Z}[x_1,\ldots,x_m, x'_1,\ldots, x'_m]/_{\mathcal{I}+\mathcal{J}}$,
\end{center} 
where $\mathcal{I}$ is the Stanley-Reisner ideal of the complex $\mathrm{Bier}(\sK)$, which has the following form: 
\[
\mathcal{I}=\mathcal{I}_\sK+\mathcal{I}_{\sK^\vee}+\langle x_ix'_i, i=1,\ldots,m\rangle
\] 
with $\mathcal{I}_\sK$ and $\mathcal{I}_{\sK^\vee}$ being the Stanley--Reisner ideals of $\sK$ and its Alexander dual $\sK^\vee$, and $\mathcal{J}$ is the ideal of linear dependencies of columns of the canonical characteristic matrix $\Lambda_\sK$. By the definition of $\Lambda_\sK$, the ideal $\mathcal J$ yields the following relations among the cohomology ring generators:
\begin{center} 
$x_i-x'_i=x_m-x'_m,\quad i=1,\ldots,m-1$. 
\end{center}

The toric manifold $X_\sK$ is a nonsingular complete algebraic variety equipped with the canonical complex structure. The natural isomorphism of Bier spheres of $\sK$ and $\sK^\vee$ induces a biholomorphic equivalence $\varphi\colon X_\sK\rightarrow X_{\sK^\vee}$ which comes from the symmetry of the corresponding regular fans $\Sigma_{\sK^\vee}=-\Sigma_\sK$. The induced isomorphism in cohomology $\varphi^\ast\colon H^\ast(X_{\sK^\vee};\mathbb{Z})\rightarrow H^\ast(X_\sK;\mathbb{Z})$ acts as the identity map. 

If $\sK$ has a ghost vertex $v\in [m]$, then $[m]\setminus\{v\}\in \sK^\vee$. If additionally $v$ is a ghost vertex of $\sK^\vee$, then the Bier sphere of $\sK$ is isomorphic to the boundary of the cross-polytope $\partial\diamondsuit^{m-1}$ which yields 
$X_\sK\cong(\mathbb{C}P^1)^{m-1}$ and $X^\mathbb{R}_\sK\cong(\mathbb{R}P^1)^{m-1}$. Henceforth, we may assume that $\sK$ is without ghost vertices.
 
Finally, the real toric manifold $X^\mathbb{R}_\sK$ is the real locus of the toric manifold $X_\sK$ and is equipped with the canonical smooth structure that turns it into a smooth $(m-1)$-dimensional manifold. The mod 2 cohomology ring of $X_\sK^\mathbb{R}$ is represented by the same formula as in the integer case, but with reduced coefficients:

 \begin{center} 
 $H^*(X_\sK^\mathbb{R};\mathbb{Z}_2)\cong\mathbb{Z}_2[x_1,\ldots,x_m, x'_1,\ldots, x'_m]/_{\mathcal{I}+\mathcal{J}}.$
 \end{center}

For the explicit computation of the rational Betti numbers of $X_\sK^{\R}$ and the partial description of the multiplicative structure of the cohomology ring $H^*(X_\sK^{\R};\Q)$, see~\cite{CYY}. 

\section{Characteristic numbers}

In this section, we prove general formulas expressing characteristic numbers of real and complex canonical toric manifolds in combinatorial terms and discuss some applications of our computations. 

\subsection{Chern numbers}

The total Chern class of the canonical toric manifold $X_\sK$ is determined by
\[c(X_\sK)=\prod_{i=1}^m(1+x_i+x'_i).\] 
Therefore, 
\[
c_k(X_\sK)=\sum_{\substack{\sigma\in \sK,\tau\in \sK^\vee \\ \sigma\cap\tau=\emptyset, |\sigma\cup\tau|=k}}x_\sigma x'_\tau,
\] 
which is the sum over all $(k-1)$-faces of $\mathrm{Bier}(\sK)$. In particular, the Euler characteristic of $X_\sK$ is equal to the number of facets of $\mathrm{Bier}(\sK)$:
\[
\chi(X_\sK)=c_{m-1}[X_\sK]=f_{m-2}(\mathrm{Bier}(\sK)).
\]
We calculate all the Chern numbers of the canonical toric manifold $X_\sK$. Since the Chern classes of $X_\sK$ are symmetric functions of two types of variables $c_k=\sigma_k(x_1,\ldots,x_m,x'_1,\ldots,x'_m), k=1,\ldots,m-1$, they can be represented in symmetric functions of these variables separately 
\begin{equation}\label{separated} c_k=\sum_{r=0}^k\sigma_r\sigma'_{k-r}, \ \  k=1,\ldots, m-1,\end{equation} where 
$\sigma_0=\sigma'_0=1, \sigma_k=\sigma_k(x_1,\ldots,x_m), \sigma'_k=\sigma_k(x'_1,\ldots,x'_m), k=1,\ldots,m-1$.
According to the ideal $\mathcal{J}$ of linear relations, we may introduce a new variable \[u=x'_i-x_i, \ \ i=1,\ldots,m.\] For $\tau=\{i_1,\ldots,i_k\}\notin \sK^\vee$ we have
\[x'_\tau=\prod_{j=1}^k(x_{i_j}+u)=u^k+\sigma_1|_{\tau}u^{k-1}+\ldots+\sigma_{k-1}|_{\tau}u+\sigma_k|_{\tau},\] where $\sigma_j|_{\tau}=\sigma_j(x_{i_1},\ldots,x_{i_k})$. Hence, the cohomology ring of $X_\sK$ is isomorphic to
\[H^\ast(X_\sK,\mathbb{Z})\cong\mathbb{Z}[x_1,\ldots,x_m, u]/_{\mathcal{I}'+\mathcal{J}'},\] where \[\mathcal{I}'=\mathcal{I}_\sK+\langle u^{|\tau|}+\sigma_1|_{\tau}u^{|\tau|-1}+\ldots+\sigma_{|\tau|-1}|_{\tau}u+\sigma_{|\tau|}|_{\tau}, \tau\notin \sK^\vee\rangle\] and \[\mathcal{J}'=\langle x_i^2+ux_i, i=1,\ldots m\rangle.\]

Symmetric functions $\sigma'_k$ are in translated variables and therefore are expressed as
\begin{equation}\label{translated}\sigma'_k=\sum_{j=0}^k{m-j \choose k-j}\sigma_j u^{k-j}, \ k=1,\ldots,m-1.\end{equation} For a partition $i_1+\cdots+i_p=m-1$ the following identity holds from $(\ref{separated})$ and $(\ref{translated})$
\[c_{i_1}\cdots c_{i_p}=\sum_{\substack{0\leq r_1\leq i_1 \\ \cdots \\ 0\leq r_p\leq i_p}}\sigma_{r_1}\cdots\sigma_{r_p}\sigma'_{s_1}\cdots\sigma'_{s_p}=\]\[=\sum_{\substack{0\leq r_1\leq i_1 \\ \cdots \\ 0\leq r_p\leq i_p}}\sum_{\substack{0\leq j_1\leq s_1 \\ \cdots \\ 0\leq j_p\leq s_p}}{m-j_1 \choose s_1-j_1}\cdots{m-j_p \choose s_p-j_p}\sigma_{r_1}\cdots\sigma_{r_p}\sigma_{j_1}\cdots\sigma_{j_p}u^{s_1+\cdots+s_p-(j_1+\cdots+j_p)},\] where $r_1+s_1=i_1,\ldots,r_p+s_p=i_p$.  Denote \[\sigma_\lambda=\sigma_{r_1}\cdots\sigma_{r_p}\sigma_{j_1}\cdots\sigma_{j_p},\] for a partition $\lambda\vdash r_1+\cdots+r_p+j_1+\cdots+j_p=|\lambda|$. Let $(M_{\lambda\mu})$ be the transition matrix from the monomial to elementary symmetric functions \[\sigma_\lambda=\sum_{\mu\vdash|\lambda|}M_{\lambda\mu}m_\mu.\] The coefficients $M_{\lambda\mu}$ have a precise combinatorial interpretation as the number of $0-1$ matrices whose columns and rows have sums of elements equal to components of the partitions $\lambda$ and $\mu$ respectively. If $\mu=1^{d_1^\mu}2^{d_2^\mu}\cdots$ and $l(\mu)$ is the length of the partition $\mu$, then according to the relations in the ideal $\mathcal{J}'$ we obtain the following
\[\sigma_{\lambda}=\sum_{\mu\vdash|\lambda|}M_{\lambda\mu}{l(\mu) \choose d_1^\mu,d_2^\mu,\cdots}\sigma_{l(\mu)}(-u)^{|\mu|-l(\mu)},\] which leads to the formula for $c_I=c_{i_1}\cdots c_{i_p}$
\[c_I=\sum_{\substack{\lambda\vdash r_1+\cdots+r_p+j_1+\cdots+j_p \\ 0\leq r_1\leq i_1,\ldots,0\leq r_p\leq i_p \\ 0\leq j_1\leq s_1,\ldots,0\leq j_p\leq s_p}}\prod_{t=1}^p{m-j_t \choose s_t-j_t}\sum_{\mu\vdash|\lambda|}(-1)^{|\mu|-l(\mu)}M_{\lambda\mu}{l(\mu) \choose d_1^\mu,d_2^\mu,\ldots}\sigma_{l(\mu)} u^{m-1-l(\mu)}.\]
Rearranging the summation, we obtain the following identity 
\begin{equation}\label{chernnumbers}
c_I=\sum_{k=0}^{m-1}\gamma_{I,k}\sigma_ku^{m-1-k},
\end{equation} where the coefficients are defined purely combinatorially 
\begin{equation}\label{defgamma} 
\gamma_{I,k}=\sum_{\substack{\mu\colon l(\mu)=k, \\ |\mu|\leq m-1}}(-1)^{|\mu|-k}{k \choose d_1^\mu, d_2^\mu,\cdots}\sum_{\substack{\lambda\vdash r_1+\cdots+r_p+j_1+\cdots+j_p=|\mu| \\ 0\leq r_1\leq i_1,\ldots,0\leq r_p\leq i_p \\ 0\leq j_1\leq s_1,\ldots,0\leq j_p\leq s_p}}M_{\lambda\mu}\prod_{t=1}^p{m-j_t \choose s_t-j_t}.
\end{equation}
Thus, there are universal vectors indexed by partitions
\[\gamma_I=(\gamma_{I,0},\gamma_{I,1},\ldots,\gamma_{I,m-1})^t, \ I\vdash m-1.\]
Up to now we only used the linear relations in the cohomology ring together with missing all edges connecting opposite vertices in the Bier sphere. The combinatorics of the simplicial complex appears through the identities
\begin{equation}\label{combinatorialchernclass}
\sigma_k=\sum_{S\in \sK, |S|=k}x_S,\ \ k=1,\ldots,m-1.
\end{equation} We have to calculate \[x_S u^{m-1-|S|}[X_\sK], S\in \sK.\] For a given face $S\in \sK$ and a vertex $i_0\in[m]$ such that $i_0\notin S$ we have
\[x_S u^{m-1-|S|}=x_S\prod_{i_0\neq i\notin S}(x'_i-x_i)=x_S\sum_{\substack{A\cup S\in \sK\setminus i_0 \\ B\in \sK^\vee\setminus i_0 \\ S\sqcup A\sqcup B=[m]\setminus\{i_0\}}}(-1)^{|A|}x_Ax'_B.\] If $S$ is a facet, then $x_Su^{m-1-|S|}[X_\sK]=1$ since $(S,i_0,[m]\setminus(S\cup\{i_0\})$ is a facet of $\mathrm{Bier}(\sK)$. For $S$ not being a maximal face, we choose $i_0\notin S$ such that $S\cup\{i_0\}\in \sK$. The above identity gives 
\[x_Su^{m-1-|S|}[X_\sK]=\sum_{\substack{A\sqcup S\in \sK\setminus i_0 \\ A\sqcup S\sqcup\{i_0\}\notin \sK}}(-1)^{|A|}=\sum_{A\in(\mathrm{link}_{\sK}S\setminus i_0)\setminus\mathrm{link}_{\mathrm{link}_{\sK}S}\{i_0\}}(-1)^{|A|}.\] Therefore, 
\[x_Su^{m-1-|S|}[X_\sK]=\chi(\mathrm{link}_{\mathrm{link}_{\sK}S}\{i_0\})-\chi(\mathrm{link}_{\sK}S\setminus i_0).\] By applying the additivity of Euler characteristic,
\[\chi(\sK\setminus v)+1-\chi(\mathrm{link}_{\sK}v)=\chi(\sK)\] to $\mathrm{link}_{\sK}S$, we obtain the following formula.

\begin{lemma}\label{face}
For any face $S\in \sK$, it holds \[x_Su^{m-1-|S|}[X_\sK]=1-\chi(\mathrm{link}_{\sK}S).\]
\end{lemma}
Note that the lemma covers the case, where $S$ is a facet, since in that case $\mathrm{link}_{\sK}S=\emptyset$. For $S$ being the empty face, we assume $\mathrm{link}_{\sK}\emptyset=K$. The lemma also shows that our choice of vertex $i_0$ is not relevant.

According to $(\ref{combinatorialchernclass})$, the following numbers are associated with each simplicial complex $\sK$ on $[m]$ with $m\geq 2$:
\[
\alpha_k(\sK)=\sigma_ku^{m-k-1}[X_{\sK}]=\sum_{\substack{S\in \sK, \\ |S|=k}}\Big(1-\chi(\mathrm{link}_{\sK}S)\Big), \ \ k=0,1,\ldots,m-1.
\] 

\begin{definition}
Given a simplicial complex $\sK$ on $[m]$ with $m\geq 2$, the integer vector 
\[
\alpha(\sK)=(\alpha_0(\sK),\alpha_1(\sK),\ldots,\alpha_{m-1}(\sK))^t
\] 
we call the \emph{$\alpha$-vector} of the complex $\sK$.
\end{definition}

Once we know the Chern classes of $[X_\sK]$, we can calculate the Chern numbers and determine the complex bordism class of the toric manifold $X_\sK$.
Given a simplicial complex $\mathcal K$, we denote by $f(\mathcal K)$ its \emph{$f$-vector}. That is, $f_{i}(\sK)$ equals the number of faces of $\sK$ of dimension $i$, where $-1\leq i\leq \dim \sK + 1$. Note that $f_{-1}(\sK)=1$ for any simplicial complex $\sK$.

\begin{proposition}\label{complexcobordism}
Let $\sK_1$ and $\sK_2$ be two simplicial complexes  with $f(\sK_1)=f(\sK_2)$. Then $[X_{\sK_1}]=[X_{\sK_2}]$ in the unitary bordism ring $\Omega^U_*$.
\end{proposition}
\begin{proof}
According to $(\ref{chernnumbers})$, the Chern numbers of $X_\sK$ depend on the complex $\sK$ in the following way
\begin{equation}\label{alpha} 
c_I[X_\sK]=\langle\alpha(\sK),\gamma_{I}\rangle=\sum_{k=0}^{m-1}\alpha_k(\sK)\gamma_{I,k}, \ I\vdash m-1.
\end{equation}  
By the standard enumerative argument, we obtain ($\alpha_i:=\alpha_i(\sK), f_i:=f_i(\sK)$):
\begin{equation}\label{alphaExplicitFormula} 
\alpha_k=(-1)^k\sum_{j=k}^{m-1}(-1)^jf_{j-1}{j \choose k}, \ \ k=0,1,\ldots,m-1.
\end{equation}
This finishes the proof. \end{proof}

\begin{remark}
The $\alpha$-vector was introduced in~\cite{JM} under the name \emph{e-vector}. There, it was shown that 
\begin{itemize}
\item $\sK$ satisfies the Dehn-Sommerville relations $h_k=h_{n-k}$ for all $0\leq k\leq n$, where, by definition,
$$
\sum\limits_{k=0}^{n}h_kt^{n-k}:=\sum\limits_{k=0}^{n}f_{n-1-k}(t-1)^k
$$
if and only if 
$$
\alpha_k=(-1)^{n+k}f_{k-1}\text{ for all }0\leq k\leq n;
$$
\item Denote by $E(M;x_1,\ldots,x_m):=\sum\limits_{\bf{a}\in\N^m}\dim_{\Bbbk}(M_{\bf{a}})\frac{\bf{x^a}}{\bf{a!}}$ the \emph{exponential Hilbert series} of a $\N^m$-graded $\Bbbk[x_1,\ldots,x_m]$-module $M$. Then one has:
$$
E(\Bbbk[\sK];t)=\sum\limits_{\sigma\in \sK}(e^{t}-1)^{|\sigma|}=\alpha_0+\alpha_1e^t+\alpha_2e^{2t}+\alpha_{n}e^{nt}.
$$ 
\end{itemize}
\end{remark}

It is easy to see that the explicit formulae for the $\alpha$-vector components obtained in the above proof are equivalent to the following identity:
$$
\sum\limits_{p=0}^{m-1}f_{p-1}(t-1)^{p}=\sum\limits_{k=0}^{m-1}\alpha_{k}t^k.
$$ 
It follows that the  $\alpha$- and $f$-vectors of any simplicial complex $\sK\neq\Delta_{[m]}$ on $[m]$ with $m\geq 2$ determine each other.

\begin{example}
If $\Gamma$ is a graph, i.e. a simplicial complex of dimension $\leq 1$, we obtain the following formula
\[c_I[X_\Gamma]=\gamma_{I,0}(1-\chi(\Gamma))+\gamma_{I,1}(v-\sum_{v\in\Gamma}d_v)+\gamma_{I,2}e.\] Since the sum of degrees $d_v$ of vertices is twice the number of edges $e$, two graphs $\Gamma_1$ and $\Gamma_2$ produce the unitary bordant canonical toric manifolds $X_{\Gamma_1}$ and $X_{\Gamma_2}$ if and only if the graphs have the same numbers of vertices and edges.
\end{example}

The same observation as for $(\ref{chernnumbers})$ shows that
\[c_I=\sum_{k=0}^{m-1}\gamma_{I,k}\sigma_k'v^{m-k-1},\] where $v=-u$.
From $(\ref{translated})$ it follows
\[c_I=\sum_{k=0}^{m-1}\gamma_{I,k}(-1)^{m-k-1}\sum_{j=0}^k{m-j\choose k-j}\sigma_ju^{m-j-1},\]
This leads us to the following numbers associated with a simplicial complex $\sK$ on $[m]$ with $m\geq 2$:
\[
\mu_k(\sK)=(-1)^{m-k-1}\sum_{j=0}^k{m-j\choose k-j}\alpha_j(\sK), \ k=0,1,\ldots,m-1.
\]

\begin{definition}
Given a simplicial complex $\sK$ on $[m]$ with $m\geq 2$, the vector 
\[
\mu(\sK)=(\mu_0(\sK),\mu_1(\sK),\ldots,\mu_{m-1}(\sK))^t
\] 
we call the \emph{$\mu$-vector} of the complex $\sK$.
\end{definition}

By definition, the $\alpha$- and $\mu$-vectors of $\sK$ are related by a linear transformation 
\[
\mu=A\alpha,
\] 
where $A$ is the lower triangular binomial $m\times m$ matrix with entries 
\[
a_{k,j}=(-1)^{m-k-1}{m-j\choose k-j}, \ 0\leq j\leq k\leq m-1.
\] 
Observe that the linear transformation $A$ is an involution $A^2=E$, and therefore we also have:
\[
\alpha=A\mu.
\]

Using formula~\eqref{alphaExplicitFormula}, we obtain the following explicit formula for the $\mu$-numbers
\[
\mu_k(\sK)=(-1)^{m-k-1}\sum_{j=0}^{m-k}(-1)^jf_{j-1}{m-j\choose k}, \ 0\leq k\leq m-1.
\]

\subsection{$\alpha$- and $\mu$-vectors}

For a simplicial complex $\sK\neq\Delta_{[m]}$ on $[m]$ with $m\geq 2$, define 
\[
E_{\sK}(x_1,\ldots,x_m)=\sum_{\sigma\in \sK}\prod_{i\in\sigma}(e^{x_i}-1),
\]
\[
M_{\sK}(x_1,\ldots,x_m)=e^{x_1+\cdots+x_m}-\sum_{\sigma\in \sK}\prod_{i\notin\sigma}(e^{x_i}-1).
\]

\begin{lemma}
For any simplicial complex $\sK\neq\Delta_{[m]}$ on $[m]$ with $m\geq 2$, one has:
$$
M_{\sK}(x_1,\ldots,x_m)=E_{\sK^\vee}(x_1,\ldots,x_m).
$$
\end{lemma}
\begin{proof}
The following chain of identities holds:
$$
M_x(\sK^\vee)=e^x-\sum\limits_{\sigma\in \sK^\vee}\prod(\sigma^c)=e^x-\sum\limits_{\tau\notin \sK}\prod(\tau)=e^x-\sum\limits_{\tau\in 2^m}\prod(\tau)+\sum\limits_{\tau\in \sK}\prod(\tau)=\sum\limits_{\tau\in \sK}\prod(\tau)=E_x(\sK),
$$
where we used the notations
$$
e^x=\prod\limits_{i=1}^{m}e^{x_i},
E_x(\sK)=E_{\sK}(x_1,\ldots,x_m),M_x(\sK)=M_{\sK}(x_1,\ldots,x_m),
\prod(\sigma)=\prod\limits_{i\in\sigma}(e^{x_i}-1),
\sigma^c=[m]\setminus \sigma.
$$
This finishes the proof.
\end{proof}

Specialize at $x_1=\cdots=x_m=t$ to define: 
\[
E_{\sK}(t)=\sum_{i=0}^{m-1}f_{i-1}(e^t-1)^i,
\]
\[
M_{\sK}(t)=e^{mt}-\sum_{i=0}^{m-1}f_{i-1}(e^t-1)^{m-i}.
\]
The binomial expansions of the summands in $E_{\sK}(t)$ give \cite{JM}
\[
E_{\sK}(t)=\alpha_0+\alpha_1e^t+\cdots+\alpha_{m-1}e^{(m-1)t}.
\] 
Using the generating function for the Stirling numbers of the second kind, we obtain the following identities 
\[
\sum_{j=0}^{m-1}j^n\alpha_j=\sum_{j=0}^{m-1}j!f_{j-1}\genfrac{\{}{\}}{0pt}{}{n}{j}, \ \  n\geq0.
\]

The following two statements are immediately verified by direct computations.

\begin{proposition}\label{EandMfunctionsSumsToOneProp}
The function $M_{\sK}(t)$ is the exponential generating function for the $\mu$-coefficients
\[
M_{\sK}(t)=\mu_0+\mu_1e^t+\cdots+\mu_{m-1}e^{(m-1)t}.
\] 
Let $\frac{1}{e^t-1}=e^s-1$. Then the exponential generating functions for the $\alpha$- and $\mu$-numbers of a simplicial complex $\sK$ are related by 
\[
e^{-ms}E_{\sK}(s)+e^{-mt}M_{\sK}(t)=1.
\]
\end{proposition}

\begin{proposition}\label{AlphaAndMuAreDualProp}
For a simplicial complex $\sK\neq\Delta_{[m]}$ on $[m]$ with $m\geq 2$, it holds: 
\[
M_{\sK}(t)=E_{\sK^\vee}(t).
\] 
In particular, 
\[
\mu_i(\sK)=\alpha_i(\sK^\vee), \ i=0,1,\ldots,m.
\]
\end{proposition}

\subsection{Unitary bordism classes of canonical manifolds}
We are going to determine the complex bordism class of a canonical toric manifold $X_\sK$ in terms of the face vector of $\sK$. Recall that two stably complex manifolds, $M_1$ and $M_2$, are bordant in the unitary bordism ring $\Omega^U_*$ if and only if they have identical
Chern characteristic numbers. The identity of Chern numbers $c_I[X_\sK]=c_I[X_{\sK^\vee}], \ I\vdash m-1$ implies the following
\[
\sum_{k=0}^{m-1}\alpha_k(\sK)\gamma_{I,k}=\sum_{k=0}^{m-1}\mu_k(\sK)\gamma_{I,k},\] which shows that the converse of Proposition \ref{complexcobordism} does not hold. 

The combinatorial nature of the coefficients of the transformation matrix $A$ between $\alpha$- and $\mu$-numbers is given by the following collection of simplicial complexes. Let 
\begin{equation}\label{complexes}
    \sK_{j}:=\Delta_{[j]}\ast\partial\Delta_{[m-j]}, \ \ j=0,1,\ldots,m-1,
\end{equation} 
then
\[
A=\begin{bmatrix}
    \alpha_0(\sK_0) & 0 & \cdots & 0\\
    \alpha_1(\sK_0) & \alpha_1(\sK_1) & \cdots & 0\\
    \vdots & \vdots & \ddots & \vdots\\
    \alpha_{m-1}(\sK_0)& \alpha_{m-1}(\sK_1)& \cdots & \alpha_{m-1}(\sK_{m-1})
\end{bmatrix}.
\]

Denote by 
$$
X_j:=X_{\sK_j}, \ j=0,1,\ldots,m-1
$$ 
the canonical toric manifolds associated with simplicial complexes in the collection $(\ref{complexes})$.

\begin{proposition}\label{XbasisProp}
The canonical toric manifold $X_j$ is diffeomorphic to the product of projective spaces
\[
X_j=(\mathbb{C}P^1)^j\times\mathbb{C}P^{m-j-1} \text{ for all } j=0,1,\ldots,m-1.
\]
\end{proposition}
\begin{proof} The Bier sphere of the cone $C\sK$ of a simplicial complex $\sK$ is the suspension $\mathrm{Bier}(C\sK)=\Sigma\mathrm{Bier}(\sK)$, which implies 
\[
X_{C\sK}=\mathbb{C}P^1\times X_\sK.
\] 
It remains to recall that $\sK_j=C^j\partial\Delta_{[m-j]}$ and $X_{\partial\Delta_{[m-j]}}=\mathbb{C}P^{m-j-1}$.
\end{proof}

The following example shows that the almost complex structure of canonical manifolds $X_j, \ j=0,1,\ldots,m-1$ is not the standard product structure of projective spaces.
\begin{example}Take $m=4$ and $X_1$ with $\sK_1=C\partial\Delta_{[3]}$, which is diffeomorphic to the product $\mathbb{C}P^1\times\mathbb{C}P^2.$ We have the following
\[c(\mathbb{C}P^1\times\mathbb{C}P^2)=(1+2x)(1+3y+3y^2),\] 
\[c_1^3[\mathbb{C}P^1\times\mathbb{C}P^2]=(2x+3y)^3[X_1]=54xy^2[X_1]=54.\] On the other hand,
\[
c_1(X_1)=\sigma_1+\sigma'_1=2\sigma_1+4u,
\] 
which gives, using $\sigma_1^2=2\sigma_2-u\sigma_1, \sigma_1\sigma_2=3\sigma_3-2\sigma_2u$, the following expression for the Chern class:
\[
c_1(X_1)^3=64u^3+56\sigma_1u^2+48\sigma_2u+48\sigma_3.
\] 
Therefore, $\gamma_{111}=(64,56,48,48)^t$. The matrix $A^t$ is given by $A^t=
\begin{bmatrix}
    -1 & 4 & -6 & 4\\
     0 & 1 & -3 & 3\\
     0 & 0 & -1 & 2\\
     0 & 0 & 0 & 1
\end{bmatrix}$. We see that $A^t\gamma_{111}=\gamma_{111}$. The same result is obtained in the dual settings $\sK_1^\vee=\Delta[1]$ with $c_1(X_1)^3=64v^3+56\sigma_1'v^2+48\sigma_2'v+48\sigma_3'$. Our calculations show that the manifolds $X_1$ and $\mathbb{C}P^1\times\mathbb{C}P^2$ have different almost complex structure since 
\[
c_1^3[\mathbb{C}P^1\times\mathbb{C}P^2]=54\neq 56=\gamma_{111,1}=c_1^3[X_1].
\]
\end{example}

The next proposition provides a geometric interpretation of the coefficients $\gamma_{I,j}$.

\begin{proposition}\label{gammaoverc}
The coefficients $\gamma_{I,j}$ are determined by 
\[
\gamma_{I,j}=(-1)^{m-1}\sum_{k=j}^{m-1}(-1)^k{m-j \choose k-j}c_I[X_k], \ \ I\vdash m-1, j=0,1,\ldots,m-1.
\]
\end{proposition}
\begin{proof}Define the vectors $c_I=(c_I(X_0),c_I(X_1),\ldots,c_I(X_{m-1}))^t, \ I\vdash m-1$. It follows from equations $(\ref{alpha})$ that $c_I=A^t\gamma_I$. Since $A^t$ is an involution, we obtain $\gamma_I=A^tc_I$. 
\end{proof}

Applying the duality $c_I[X_j]=c_I[X_j^\vee], \ I\vdash m-1$ for the simplicial complexes $\sK_j, \ j=0,1,\ldots,m-1$, we immediately obtain the following relations between the $\gamma$-coefficients
\begin{equation}\label{gammadependence}
\gamma_{I,j}=\sum_{k=j}^{m-1}(-1)^{m-k-1}{m-j\choose k-j}\gamma_{I,k}, \ j=0,1,\ldots,m-1,
\end{equation} 
since $\alpha_k(\sK_j^\vee)=\delta_{k,j}$. It follows that the vectors $\gamma_I$ are the eigenvectors of the matrix $A^t$: 
\[
A^t\gamma_I=\gamma_I, \ I\vdash m-1.
\] 
This recovers the topological nature of the $\gamma$-coefficients introduced in the next result. 

\begin{proposition}\label{topologicalgamma}
The $\gamma$-coefficients are the Chern numbers of the canonical toric manifolds $X_j$:
\[
\gamma_{I,j}=c_I[X_j], \ j=0,1,\ldots,m-1, \ I\vdash m-1.
\]
\end{proposition}

Propositions~\ref{gammaoverc} and \ref{topologicalgamma} yield the following relations in the unitary bordism ring.

\begin{proposition}\label{RelationsXJmanifoldsCoro}
The space of relations between the classes $[X_p], 0\leq p\leq m-1$ in $\Omega^U_{2(m-1)}$ is generated by  \[
[X_j]=\sum_{k=j}^{m-1}(-1)^{m-k-1}{m-j\choose k-j}[X_k], \ j=0,1,\ldots,m-1. 
\]  
\end{proposition}
\begin{example}
Let $m=4$. Then the relations between $[X_j]$ for $0\leq j\leq 3$ can be written as follows:
$$
[X_0]=2[X_1]-[X_2]\text{ and }[X_2]=[X_3].
$$
Take $\sK=\langle\{1,2,3\},\{4\}\rangle$. Then $\alpha(\sK)=(-1,1,0,1)$ and $\mu(\sK)=(1,-3,3,0)$, and we obtain: 
$$
[X_\sK]=-[X_0]+[X_1]+[X_3]=[X_0]-3[X_1]+3[X_2].
$$
\end{example}

The Proposition \ref{topologicalgamma} and the formula $(\ref{alpha})$ show that any complex bordism class $[X_\sK]$ lies in the integer span of complex bordism classes $[X_k], \ k=0,1,\ldots,m-1$. More precisely, the following statement holds.

\begin{theorem}\label{additive}
The next identity takes place in the unitary bordism group $\Omega^U_{2m-2}\colon$
\[
[X_{\sK}]=\sum_{j=0}^{m-1}\alpha_j(\sK)[X_{j}].
\] 
The rank of the free abelian subgroup in $\Omega^{U}_{2m-2}$ generated by all the unitary bordism classes of $(2m-2)$-dimensional canonical toric manifolds is equal to $\lfloor\frac{m+1}{2}\rfloor$. The set $\{[X_{m-1-2k}]\mid 0\leq k\leq \lfloor\frac{m-1}{2}\rfloor\}$ serves as a basis for that subgroup.
\end{theorem}
\begin{proof}
It remains to observe that the last part is a consequence of the identity 
$\dim\ker(A^t-E)=\lfloor\frac{m+1}{2}\rfloor$. 
\end{proof}

As we mentioned above, the converse statement to that of Proposition~\ref{complexcobordism}  does not hold. However, here is the necessary and sufficient condition for the two canonical toric manifolds to be bordant in $\Omega^U_*$. 
\begin{theorem}\label{eqcomplexcobordism}
Let $\sK_1$ and $\sK_2$ be two simplicial complexes on $[m], \ m\geq 2$ not equal to $\Delta_{[m]}$. Then one has:
$$
[X_{\sK_1}]=[X_{\sK_2}]\text{ in }\Omega^U_* \ \text{if and only if} \ f_{i-1}(\sK_1)-f_{i-1}(\sK_2)=f_{m-i-1}(\sK_1)-f_{m-i-1}(\sK_2), \ 1\leq i\leq\lfloor\frac{m-1}{2}\rfloor.
$$ 
\end{theorem}
\begin{proof}
For any simplicial complex $\sK\neq\Delta_{[m]}$ on $[m]$ with $m\geq 2$, we have the following identity due to $(\ref{chernnumbers})$:  
\[
c_I[X_{\sK}]=\langle \alpha(\sK),\gamma_I\rangle, \ I\vdash m-1.
\] 
By equations $(\ref{gammadependence})$, the vectors $\gamma_I, \ I\vdash m-1$ span the $(+1)-$eigenspace $E_+(A^t)$ of the matrix $A^t$.  Since $A$ is an involution, it follows from 
$E_+(A^t)^\perp=E_-(A)$ that \[[X_{\sK_1}]=[X_{\sK_2}] \ \text{if and only if} \ A\alpha(\sK_1)-A\alpha(\sK_2)=\alpha(\sK_2)-\alpha(\sK_1).\] The relation $\mu(\sK)=A\alpha(\sK)$ together with 
Proposition~\ref{AlphaAndMuAreDualProp} implies 
\[[X_{\sK_1}]=[X_{\sK_2}] \ \text{if and only if} \ \alpha(\sK_1)+\alpha(\sK_1^\vee)=\alpha(\sK_2)+\alpha(\sK_2^\vee).\] Applying the transformation matrix $B$ relating the $\alpha$- and $f$-vectors $(\ref{alphaExplicitFormula})$, which is also an involution, gives 
\[[X_{\sK_1}]=[X_{\sK_2}] \ \text{if and only if} \ f(\sK_1)+f(\sK_1^\vee)=f(\sK_2)+f(\sK_2^\vee),\] which finishes the proof, since $f_{i-1}(\sK^\vee)={m\choose i}-f_{m-i-1}(\sK), \ 0\leq i\leq m-1.$
\end{proof}

\begin{example}
For $m=4$, we obtain $[X_{\sK}]=[X_{\sK'}]$ in $\Omega_{6}^U$ if and only if 
\[f_0-f_0'=f_2-f_2'.\] The complexes $\sK=\{1,2,3\}$ and $\sK'=\{123,4\}$ satisfy this condition.
For $m=5$, the relations on $f$-vectors are \[f_0-f_0'=f_3-f_3', \ f_1-f_1'=f_2-f_2'.\] These are satisfied by $\sK=\{1,23,34,35,45\}$ and $\sK'=\{1,234,35,45\}$, so $[X_{\sK}]=[X_{\sK'}]$ in $\Omega_{8}^U$. Note that both examples are pairs of complexes which are neither isomorphic nor dual to each other.
\end{example} 

\subsection{$\chi_y$-genus}
The $\chi_y$-genus is the complex cobordism genus determined by the Hodge numbers. For a toric manifold, it is completely determined by the $h$-vector of the underlying fan, that is, by the Betti numbers of the manifold. Therefore, the $\chi_y$-genus is the smooth invariant of toric manifolds, rather than a complex invariant. For a canonical toric manifold $X_{\sK}$, we have the following well-known formula:
\begin{equation}\label{chiyoverh}
\chi_y(X_{\sK})=\sum_{p=0}^{m-1}h_p(\mathrm{Bier}(\sK))(-y)^p.
\end{equation}
Recall that 
\[
h_p(\mathrm{Bier}(\sK))=\beta_{2p}(X_\sK)=h^{p,p}(X_\sK).
\]
Since the canonical manifolds $X_j$ are diffeomorphic to the products of projective spaces by Proposition \ref{XbasisProp}, we have the following
\begin{equation}\label{chiy}
\chi_y(X_{\sK})=\sum_{k=0}^{m-1}\alpha_k(\sK)(1-y)^k\frac{1-(-y)^{m-k}}{1+y}.
\end{equation} 
We see that the $\chi_y$-genus of a canonical toric manifold $X_{\sK}$ satisfies 
\[\chi_y(X_\sK)=(-y)^{m-1}\chi_{1/y}(X_{\sK}),\] which is equivalent, by $(\ref{chiyoverh})$, to the Dehn-Sommerville relations for the Bier sphere $\mathrm{Bier}(\sK)$:
\[h_p(\Bier(\sK))=h_{m-p-1}(\mathrm{Bier}(\sK)), \ 0\leq p\leq m-1.\]
The equation $(\ref{chiy})$ also provides a topological proof of the formula for $h_p(\mathrm{Bier}(\sK)), \ 0\leq p\leq m-1$ in terms of the $f$-vector of $\sK$. The combinatorial proof was obtained in~\cite{BPSZ05}.

\begin{proposition}\label{TopProofHvectorBierProp}
For a simplicial complex $\sK\neq\Delta_{[m]}$ on $[m]$ with $m\geq 2$, we have $(f_i=f_i(\sK))\colon$
\[h_p(\mathrm{Bier(\sK)})=1+(f_0+\cdots+f_{p-1})-(f_{m-p-1}+\cdots+f_{m-2}), \ \ 1\leq p\leq\lfloor\frac{m-1}{2}\rfloor.\]
\end{proposition}
\begin{proof}
Substitute $\alpha_k(\sK)$ into the sum $(\ref{chiy})$ and rearrange the double sum to get: 
\[\chi_y(X_{\sK})=\sum_{i=0}^{m-1}(-1)^if_{i-1}\frac{1}{1+y}\sum_{j=0}^i{i\choose j}(y-1)^j(1-(-y)^{m-j}).\] Summing the inner binomial sum gives the following:
\[
\chi_y(X_{\sK})=\sum_{i=0}^{m-1}f_{i-1}\frac{(-y)^i-(-y)^{m-i}}{1+y}=\sum_{i=0}^{m-1}f_{i-1}\big[(-y)^i+\cdots+(-y)^{m-i-1}\big].
\] 
It remains to extract the coefficient by $(-y)^p$ to finish the proof.
\end{proof} 

The $\chi_y$-genus specializes to the classical invariants, providing their expressions in terms of $\alpha(\sK)$:
\[
\chi(X_\sK)=\chi_{-1}(X_\sK)=\sum_{p=0}^{m-1}h_p(\Bier(\sK))=\sum_{p=0}^{m-1}\alpha_p(\sK)2^p(m-p),
\]
\[
\Td(X_{\sK})=\chi_0(X_{\sK})=\sum_{p=0}^{m-1}\alpha_p(\sK)=1,
\]
\begin{equation}\label{chiysign}\mathrm{sign}(X_{\sK})=\chi_1(X_{\sK})=\sum_{p=0}^{m-1}(-1)^ph_p(\Bier(\sK))=\begin{cases}\alpha_0(\sK),& m \ \text{is odd}, \\ 0, & m \ \text{is even}.\end{cases}\end{equation}

\begin{remark}
The above signature formula confirms the known formula for the signature of the toric manifold $X_\Sigma$ arising from combinatorics of a complete, regular fan $\Sigma$ in $\mathbb{R}^d$:
\[
\mathrm{sign}(X_\Sigma)=\sum_{j=0}^{d}(-1)^jh_j,
\] 
where $h_j$ are components of the $h$-vector of the underlying simplicial complex of $\Sigma$. It follows that if $\Sigma=\Sigma_\sK$ and $m$ is even, then $d=m-1$ is odd and the right hand side is zero due to the Dehn-Sommerville equations for the Bier sphere $\Bier(\sK)$. When $m$ is odd, we get 
$$
\sign(X_\sK)=h_0(\Bier(\sK))-h_{1}(\Bier(\sK))+\ldots+h_{m-1}(\Bier(\sK))=
$$
$$
=g_0(\Bier(\sK))-g_1(\Bier(\sK))+g_2(\Bier(\sK))+\ldots+(-1)^{(m-1)/2}g_{(m-1)/2}(\Bier(\sK)),
$$ where $g_i(\mathrm{Bier}(\sK))=h_i(\Bier(\sK))-h_{i-1}(\Bier(\sK)), \ 1\leq i\leq(m-1)/2$ are components of the $g$-vector of $\Bier(\sK)$.
Due to Proposition~\ref{TopProofHvectorBierProp}, for any simplicial complex $\sK\neq\Delta_{[m]}$ on $[m]$ one has:
$$
g_{i}(\Bier(\sK))=f_{i-1}(\sK)-f_{m-i-1}(\sK) \text{ for }0\leq i\leq (m-1)/2.
$$
Hence, plugging this into the previous formula, we finally obtain:
$$
\sign(X_\sK)=1-f_0(\sK)+f_1(\sK)-f_2(\sK)+\ldots = 1-\chi(\sK),
$$
where we used the fact that $h_0=g_0=1$ for each sphere and $f_{-1}=1$ for each simplicial complex. 
\end{remark}

Applying Proposition \ref{TopProofHvectorBierProp} and Theorem \ref{eqcomplexcobordism}, we find that the bordism class of a canonical toric manifold $X_{\sK}$ is completely determined by the $h$-vector of the Bier sphere $\Bier(\sK)$.

\begin{theorem}\label{eqcomplexcobordismh}
Let $\sK_1$ and $\sK_2$ be two simplicial complexes on $[m], \ m\geq 2$ not equal to $\Delta_{[m]}$. Then one has:
$$
[X_{\sK_1}]=[X_{\sK_2}]\text{ in }\Omega^U_* \ \text{if and only if} \ \ h_i(\mathrm{Bier}(\sK_1))=h_i(\mathrm{Bier}(\sK_2)), \ 1\leq i\leq\lfloor\frac{m-1}{2}\rfloor.
$$ 
\end{theorem}
\subsection{Milnor numbers} Recall that due to the theorem of Milnor and Novikov, we have the following discription of the unitary bordism ring: 
\[
  \Omega^U_*\cong\mathbb{Z}[a_{i}\colon i\ge 1],\quad \deg{a_i}=2i.
\]
Moreover, polynomial generators $a_i$ are detected by their Milnor numbers $s_i$. For any integer $i\geq 1$, set 
\[
  m_{i}=\begin{cases}
  1&\text{if $i+1\neq p^k$ for any prime $p$;}\\
  p&\text{if $i+1=p^k$ for some prime $p$ and integer $k>0$.}
  \end{cases}
\]
Then the bordism class of a stably complex manifold $M^{2i}$ may be taken to be the $2i$-dimensional generator $a_i$ if and only if $s_{i}[M^{2i}]=\pm m_{i}$. In order to determine which canonical toric manifolds of simplicial complexes may serve as generators of the unitary bordism ring $\Omega^U_\ast$, we are going to obtain a general formula for their Milnor numbers:
\[
s_{m-1}[X_K]=\langle x_1^{m-1}+\cdots+x_m^{m-1}+{x'_1}^{m-1}+\cdots+{x'_m}^{m-1},[X_{K}]\rangle.
\]

\begin{lemma}\label{class}
The characteristic class $s_{m-1}$ depends only on $\sigma_1$ and $u$ in the following way
\[
s_{m-1}=\Big((-1)^m+1\Big)\sigma_1u^{m-2}+mu^{m-1}.
\]
\end{lemma}
\begin{proof} 
First, we note the following
\[x_1^{m-1}+\cdots+x_m^{m-1}=(-1)^m\sigma_1u^{m-2}.\] Since $x'_i=u+x_i, i=1,\ldots,m$ we have
\[{x'_i}^{m-1}=-x_iu^{m-2}\sum_{j=0}^{m-2}{m-1 \choose j}(-1)^{m-1-j}+u^{m-1}=x_iu^{m-2}+u^{m-1},\] which gives
\[{x'_1}^{m-1}+\cdots+{x'_m}^{m-1}=\sigma_1u^{m-2}+mu^{m-1}.\]
\end{proof}
The Lemma~\ref{face} and Lemma~\ref{class} together imply the following proposition.

\begin{proposition}
The Milnor number for the canonical toric manifold $X_\sK$ of a simplicial complex $\sK\neq\Delta_{[m]}$ on $[m]$ with $m\geq 2$ is given in terms of the $\alpha$-vector of $\sK$ by
\[
s_{m-1}[X_\sK]=m\alpha_0(\sK)+\Big((-1)^m+1\Big)\alpha_1(\sK).
\]
\end{proposition}

We can check that the formula for the dual complex $\sK^\vee$ gives the same result
\[s_{m-1}[X_{\sK^\vee}]=s_{m-1}[X_\sK],\] which is a necessary condition for any characteristic number.

In view of the Milnor-Novikov theorem in complex cobordism theory, by a direct computation, we obtain the conditions under which canonical toric manifolds represent polynomial generators of the unitary bordism ring.

\begin{theorem}
The class $[X_\sK]\in\Omega^U_{2(m-1)}$ of the canonical toric manifold $X_\sK$ associated with a simplicial complex $\sK$ on $[m]$ with $m\geq 2$ is a polynomial generator of the unitary bordism ring $\Omega_\ast^U$ if and only if one of the following two conditions holds:
\begin{itemize}
    \item $m=p$ is an odd prime and $\alpha_0(\sK)=\pm 1$;
    \item $m=2^k$ and $2^{k-1}\alpha_0(\sK)+\alpha_1(\sK)=\pm 1$.
\end{itemize}
\end{theorem}
\begin{example}
Let $\Sigma_m^d$ be a simplicial $d$-dimensional sphere with $m$ vertices. Then the simplicial complexes $\sK=\Sigma_p^d$ and $\sK=v\ast\Sigma_{2^k-1}^d$ satisfy the conditions of the theorem.
\end{example}

\subsection{Todd class} The total Todd class of $X_\sK$ is defined by 
\[
\mathrm{Td}=1+\Td_1+\ldots+\Td_{m-1}=\prod_{j=1}^{m}\frac{x_j}{1-e^{-x_j}}\cdot\frac{x'_j}{1-e^{-x'_j}}.
\] 
The polynomial $\mathrm{Td}_{m-1}$ is symmetric and can be expanded in Chern classes over the rational coefficients 
\[
\mathrm{Td}_{m-1}=\mathrm{Td}_{m-1}(c_1,\ldots,c_{m-1})=\sum_{I\vdash m-1}\tau_I c_I.
\] 

The Todd genus of a compact, almost complex manifold is always equal to one. Therefore, for any simplicial complex $\sK\neq\Delta_{[m]}$ on $[m]$ with $m\geq 2$, we have
\[
\mathrm{Td}_{m-1}[X_\sK]=1,
\] 
which immediately implies that $[X_\sK]\neq 0$ in $\Omega^U_*$. Furthermore, alongside with equation~\eqref{alpha} this also implies nontrivial number-theoretic identities among $\gamma$-coefficients and rational coefficients $\tau_I$ in the additive expansion of the Todd class.

\begin{proposition}\label{xi}
The following identities hold
\[\xi_k=\sum_{I\vdash m-1}\tau_I\gamma_{I,k}=1\] for all $m\geq3$ and $k=0,1,\ldots,m-1$. 
\end{proposition}
\begin{proof}
For any simplicial complex $\sK\neq\Delta_{[m]}$ on $[m]$ with $m\geq 2$, it holds by $(\ref{alpha})$ that
\[
\mathrm{Td}_{m-1}[X_{\sK}]=\sum_{k=0}^{m-1}\alpha_k(\sK)\xi_k=1.
\] 
Denote $\xi=(\xi_0,\xi_1,\ldots,\xi_{m-1})^t$ and $\mathbf{1}=(1,1,\ldots,1)^t$. Applying this identity to the collection of simplicial complexes $(\ref{complexes})$, we get the system of linear equations 
\[
A^t\xi=\mathbf{1}.
\] 
Since $A$ is an involution, we have 
\[
\xi=A^t\mathbf{1}=\mathbf{1},
\]
which finishes the proof.
\end{proof}

\begin{example}
For $m=3$ we have $\gamma_{11,0}=9, \gamma_{2,0}=3,\gamma_{11,1}=\gamma_{11,2}=8, \gamma_{2,1}=\gamma_{2,2}=4$. Since $\mathrm{Td}_2=\frac{1}{12}(c_1^2+c_2)$, we obtain 
\[\xi_k=\frac{1}{12}(\gamma_{11,k}+\gamma_{2,k})=1, k=0,1,2.\] For $m=4$ we have $\gamma_{12,0}=\gamma_{12,1}=\gamma_{12,2}=\gamma_{12,3}=24$ and $\mathrm{Td}_3=\frac{c_1c_2}{24}$, which gives \[\xi_k=\frac{1}{24}\gamma_{12,k}=1, k=0,1,2,3.\]
\end{example} 
As a corollary, we obtain a topological proof of the well-known combinatorial identity that holds for all simplicial complexes.

\begin{corollary}\label{SumOfAlphaCoro}
For any simplicial complex $\sK\neq\Delta_{[m]}$ on $[m]$ with $m\geq 2$, the following identity holds:
\[
\sum_{k=0}^{m-1}\alpha_k(\sK)=\sum_{S\in\sK}(1-\chi(\mathrm{link}_{\sK}S))=1.
\]
\end{corollary}

\subsection{Stiefel--Whitney numbers}

In this subsection, we turn to studying the canonical real toric manifolds. We will find a general formula for their Stiefel--Whitney numbers and deduce a combinatorial criterion for when the manifold $X^{\R}_\sK$ bords. The calculations become simpler in the mod 2 coefficients. Denote by $u=x_i+x'_i \mod{2}, \ \ i=1,\ldots,m$. The total Stiefel--Whitney class is given by
\begin{equation}\label{totalsw} 
w(X_\sK^\mathbb{R})=\prod_{i=1}^m(1+x_i+x'_i)=(1+u)^m.
\end{equation} Therefore, for a partition $I=(i_1,\ldots,i_p)\vdash m-1$ we have
\[w_I[X_\sK^\mathbb{R}]={m\choose i_1}\cdots{m\choose i_p}u^{m-1}[X_\sK^\mathbb{R}],\] which, together with 
\[u^{m-1}=\prod_{i=1}^{m-1}(x_i+x'_i)=\sum_{\sigma\in \sK\setminus m,\sigma\cup\{m\}\notin \sK}x_\sigma x'_{\sigma^c\setminus\{m\}},\] gives the following formula for each Stiefel--Whitney number of the smooth $(m-1)$-dimensional canonical real toric manifold: 
\[
w_I[X_\sK^\mathbb{R}]=(1+\chi(\sK))\prod_{t=1}^p{m\choose i_t}\mod{2}.
\]
This solves the problem of determining the bordism classes of real canonical toric manifolds.

\begin{theorem}
Let $X^\R_\sK$ be the canonical real toric manifold of a simplicial complex $\sK\neq\Delta_{[m]}$ on $[m]$ with $m\geq 2$. Then $X^\R_\sK$ bounds a smooth real $m$-dimensional manifold if and only if one of the following conditions holds:
\begin{itemize}
\item $m$ is even;
\item $m$ is odd and $\chi(\sK)$ is odd.
\end{itemize}
Otherwise, the canonical real toric manifold $X^\R_\sK$ is nonorientedly bordant to the projective space $\R P^{m-1}$. 
In addition, all oriented canonical real toric manifolds are orientedly null-bordant.
\end{theorem}
\begin{proof} For $m$ being even, all Stiefel--Whitney numbers vanish $w_I[X_{\sK}^\mathbb{R}]=0$. For $m$ being odd, from the binary expansion $m=2^{i_1}+\cdots+2^{i_p}+1$, for the corresponding Stiefel--Whitney number, we have
\[
w_{2^{i_1}}\cdots w_{2^{i_p}}[X_\sK^\mathbb{R}]=(1+\chi(\sK))\mod{2}.
\] 
In the case where $m$ is odd and $\chi(\sK)$ is even, we see that all Stiefel--Whitney numbers of $X_\sK^\mathbb{R}$ and $\mathbb{R}P^{m-1}$ coincide.

The manifold $X_\sK^\mathbb{R}$ is orientable if and only if $w_1(X_{\sK}^\mathbb{R})=mu=0$, that is, when $m$ is even. In that case, all Stiefel--Whitney numbers and, by dimensionality reason, all Pontryagin numbers vanish. 
\end{proof}

\subsection{The Immersion problem}

According to $(\ref{totalsw})$ the dual Stiefel--Whitney classes of the canonical real toric manifold $X_\sK^\mathbb{R}$ are obtained from
\[
\overline{w}(X_\sK^\mathbb{R})=(1+u)^{-m},
\] 
which gives 
\[
\overline{w}_k(X_\sK^\mathbb{R})={m+k-1\choose k}u^k, \ \ k=1,\ldots,m-1.
\] 

\begin{theorem}\label{thhighest}
 If the canonical real toric manifold $X_\sK^\mathbb{R}$ of a simplicial complex $\sK\neq\Delta_{[m]}$ on $[m]$ with $2^{p-1}<m\leq 2^p$ is immersed in the Euclidean space $\mathbb{R}^N$, then
 \[
 2^p-1\leq N.
 \]
Similarly, if the canonical toric manifold $X_\sK$ of a simplicial complex $\sK\neq\Delta_{[m]}$ on $[m]$ with $2^{p-1}<m\leq 2^p$ is immersed in the Euclidean space $\mathbb{R}^N$, then
\[
2(2^p-1)\leq N.
\]
\end{theorem}
\begin{proof} Let $k_{\mathrm{max}}=\max\big\{k<m:{m+k-1\choose m-1}=1\mod{2}\big\}$, which is exactly the bitwise complement of $m-1$ in the $p$ digits. Since $u^k\neq 0$ for $k<m$, we have \[k_{\mathrm{max}}=\max\{k:\overline{w}_k(X_\sK^\mathbb{R})\neq 0\}=(2^p-1)-(m-1).\]
In the complex case, the proof goes the same way. Since $w(X_\sK)=c(X_\sK) \mod{2}$, we have 
\[
w(X_\sK)=(1+u)^m,
\] 
where $u=x_i+x'_i \mod{2}, \ i=1,\ldots,m$ is the $2$-dimensional class $u\in H^2(X_\sK,\mathbb{Z}_2)$.
\end{proof}

As a corollary, we obtain a large class of real toric manifolds with the highest dimension of Euclidian space required for immersions. 

\begin{corollary}\label{highest}
For $m=2^{p-1}+1$, we have that $X_\sK^\mathbb{R}$ cannot be immersed in $\mathbb{R}^{2m-4}.$
\end{corollary}
Note that for $\sK=\partial\Delta_{[m]}$ with $m=2^{p-1}+1$ we have $X_\sK^\mathbb{R}=\mathbb{R}P^{2^{p-1}}$ and Corollary $\ref{highest}$ generalizes the standard fact about immersions of real projective spaces, see~\cite[Theorem 4.8]{MS}. 

The Cohen theorem \cite{Cohen} resolving the immersion conjecture states that the best upper bound for the dimension of a Euclidean space where a smooth, compact $n$-manifold ($n>1$) can be immersed is $2n-\alpha(n)$, where $\alpha(n)$ is the number of 1's in the binary expansion of $n$. We construct a family of real toric manifolds of arbitrary dimensions that achieve these estimates. Besides the classical example of products of real projective spaces, the corresponding examples among small covers are constructed in \cite{BG} and among real Bott manifolds in \cite{DD}. 

\begin{theorem}
For each $n>1$, there is a real toric manifold $X^\mathbb{R}_\Sigma$ of dimension $n$ such that the smallest dimension of a Euclidean space where it can be immersed is $2n-\alpha(n)$. 
\end{theorem}
\begin{proof} 
Let $m-1=2^{i_1}+\cdots+2^{i_t}, \ i_1>\ldots> i_t\geq0$ be the binary expansion of an integer $m>2$. Take arbitrary simplicial complexes $K_j, \ j=1,\ldots,t$ on $2^{i_j}+1, \ j=1,\ldots,t$ vertices, respectively, such that $K_j\neq\Delta_{[2^{i_j}+1]}, \ j=1,\ldots,m$. Then the real toric manifold for the fan $\Sigma=\Sigma_{K_1}\times\cdots\times\Sigma_{K_t}$ is the product of canonical real toric manifolds
\[X_\Sigma^\mathbb{R}=X_{K_1}^\mathbb{R}\times\cdots\times X_{K_t}^\mathbb{R}.\] If $X_\Sigma^\mathbb{R}$ is immersed in $\mathbb{R}^N$, then by Theorem \ref{thhighest} it has to be
\[
N\geq2n-\alpha(n), \  n=m-1,
\]
which finishes the proof. \end{proof}

\subsection{Pontryagin numbers} In order to describe the oriented bordism classes of canonical toric manifolds, we are going to introduce a general formula for their Pontryagin numbers. The total Pontryagin class of the canonical toric manifold $X_\sK$ depends only on the class $u$ 
\[
p(X_\sK)=\prod_{i=1}^m(1+x_i^2+{x'_i}^2)=(1+u^2)^m.
\] 
Hence, 
\[
p_k={m \choose k}u^{2k}, 1\leq k\leq\lfloor\frac{m-1}{2}\rfloor.
\] 
For an odd $m=2n+1$, the Pontryagin numbers $X_{\sK}$ are determined by
\[
p_I[X_\sK]=\prod_{t=1}^p{m \choose i_t}u^{2n}[X_{\sK}], \ I=(i_1,\ldots,i_p)\vdash n,
\] 
which by Lemma \ref{face} gives
\begin{equation}\label{pontryagin} p_I[X_{\sK}]=(1-\chi(\sK))\prod_{t=1}^p{m\choose i_t} \ I=(i_1,\ldots,i_p)\vdash n.
\end{equation}
This provides a topological proof for the signature formula of $X_{\sK}$, and reproves the formula $(\ref{chiysign})$.

\begin{proposition}
The signature of the toric manifold $X_{\sK}$ associated with a simplicial complex $\sK\neq\Delta_{[m]}$ on $[m]$ with $m\geq 2$ being odd equals
\[
\mathrm{sign}(X_{\sK})=1-\chi(\sK).
\]
\end{proposition}
\begin{proof} 
The signature of an oriented manifold $X$ of real dimension $4n$ is equal to the $L$-genus 
\[
\mathrm{sign}(X)=L_n(p_1,\ldots,p_n)[X].
\] 
Take a simplicial complex $\tilde{\sK}$ with $\chi(\tilde{\sK})=0$ on $[m]$ with $m=2n+1$. The toric manifold $X_{\tilde{\sK}}$ and the projective space $\mathbb{C}P^{2n}$ have the same Pontryagin numbers by $(\ref{pontryagin})$. Hence, 
\[
L_n({2n+1\choose 1}u^2,{2n+1\choose 2}u^4,\ldots,{2n+1\choose n}u^{2n})=u^{2n}.
\] 
Applying to the given complex $\sK$ on $[m]$ with $m=2n+1$, we get
\[
\mathrm{sign}(X_{\sK})=L_n({2n+1 \choose 1}u^2,{2n+1 \choose 2}u^4,\ldots,{2n+1 \choose n}u^{2n})[X_{\sK}]=u^{2n}[X_{\sK}]=1-\chi(\sK),
\] 
which finishes the proof.
\end{proof}

The Stiefel--Whitney classes are topological invariants. Since odd-dimensional Stiefel--Whitney classes of toric manifolds vanish, we write 
\[
w_I=w_{2i_1}\cdots w_{2i_p}=c_I\mod 2, \ I=(i_1,\ldots,i_p)\vdash m-1.
\] 
It follows from Proposition \ref{XbasisProp} that 
\[
w_I[X_j]=w_I[(\mathbb{C}P^1)^j\times\mathbb{C}P^{m-j-1}]\ \  \text{for all} \ \  I\vdash m-1 \ \ \text{and} \ \ j=0,1,\ldots,m-1.
\] 

The explicit calculation gives the following formula.

\begin{lemma}\label{gamma}
For any partition $I=(i_1,\ldots,i_p)\vdash m-1$ and $j=0,1\ldots,m-1$, we have: 
\[
c_I((\mathbb{C}P^1)^j\times\mathbb{C}P^{m-j-1})=2^j\sum_{\substack{(j_1,\ldots,j_p)\vdash j, \\ 0\leq j_t\leq i_t, t=1,\ldots,p}}{j\choose j_1,\ldots,j_p}\prod_{t=1}^p{m-j\choose i_t-j_t}.
\]
\end{lemma}
\begin{proof} The total Chern class of the product of projective spaces is 
\[c((\mathbb{C}P^1)^j\times\mathbb{C}P^{m-j-1})=(1+u_1)^2\cdots(1+u_j)^2(1+v)^{m-j},\] where $u_1^2=\cdots=u_j^2=v^{m-j}=0$. Therefore, the Chern class $c_I=c_{i_1}\cdots c_{i_p}$ is determined by
\[c_I=\prod_{t=1}^p\sum_{j_t=0}^{i_t}2^{j_t}\sigma_{j_t}(u_1,\ldots,u_j){m-j\choose i_t-j_t}v^{i_t-j_t}.\] We extract the coefficient by the monomial $u_1\cdots u_jv^{m-j-1}$ to derive the formula.
\end{proof}
\begin{theorem}
Let $\sK\neq\Delta_{[m]}$ be a simplicial complex on $[m]$ with $m\geq 2$. The canonical toric manifold $X_{\sK}$ is orientedly null-bordant if and only if one of the following conditions holds:
\begin{itemize}
    \item $m$ is even;
    \item $m$ is odd and $\chi(\sK)=1$.
\end{itemize}
If $m$ is odd and $\chi(\sK)\neq 1$ the manifold $X_{\sK}$ is nonorientedly null-bordant if and only if $\chi(\sK)$ is odd. 
Additionally, for $m$ being odd, the oriented bordism class of $X_{\sK}$ is determined by 
\[
[X_{\sK}]=(1-\chi(\sK))[\mathbb{C}P^{2n}] \ \text{in} \ \Omega^{SO}_\ast.
\]
\end{theorem}
\begin{proof} 
Let $w_{2k}(X_\sK)=c_k(X_\sK)\mod{2}, \  k=1,\ldots,m-1$ be the Stiefel--Whitney classes of $X_\sK$. From
\[
w_I[X_{\sK}]=\sum_{j=0}^{m-1}\alpha_j(\sK)w_I[X_j]
\] 
and Lemma~\ref{gamma}, the following identities can be derived:
\begin{equation}\label{stwh} w_I[X_\sK]=\alpha_0(\sK)w_I[\mathbb{C}P^{m-1}]=(1-\chi(\sK))\prod_{t=1}^p{m\choose i_t}\mod{2}, \ I=(i_1,\ldots,i_p)\vdash m-1.
\end{equation} 
The Pontryagin numbers of $X_{\sK}$ vanish $p_I(X_{\sK})=0$ in the case of even $m$ due to the dimensionality reason and for odd $m$ together with $\chi(\sK)=1$ by $(\ref{pontryagin})$. In both cases, we see that all Stiefel--Whitney numbers also vanish $w_I[X_\sK]=0$. Also, for $m$ and $\chi(\sK)\neq 1$ being odd, the Pontryagin numbers are nontrivial, but all Stiefel--Whitney numbers vanish by $(\ref{stwh})$. In the remaining case where $m$ is odd and $\chi(\sK)$ is even, at least one Stiefel--Whitney number is nonzero, since 
\[
w_{2\cdot2^{i_1}}\cdots w_{2\cdot2^{i_p}}[X_\sK]=1
\] 
for the binary expansion $m=2^{i_1}+\cdots+2^{i_p}+1$. 
The last statement follows from formulae $(\ref{pontryagin})$ and $(\ref{stwh})$.
\end{proof}

\begin{corollary}
Let  $\phi\colon\Omega_\ast^{SO}\rightarrow\Omega_\ast^O$ be the orientation forgetting homomorphism. Then, for a simplicial complex $\sK\neq\Delta_{[m]}$ on $[m]$ with $m\geq 2$ being odd and with the odd Euler characteristic $\chi(\sK)\neq 1$, we have $[X_{\sK}]\neq 0$ in $\Omega_\ast^{SO}$ and $\phi([X_{\sK}])=0.$
\end{corollary}

\end{document}